\documentclass[11pt,a4paper]{article}
\usepackage{amsfonts,amsgen,amstext,amsbsy,amsopn,amsfonts,amssymb,amscd}
\usepackage[leqno]{amsmath}
\usepackage[amsmath,amsthm,thmmarks]{ntheorem}
\usepackage{epsf,epsfig}
\usepackage{float}
\usepackage{dsfont}
\usepackage{ebezier,eepic}
\usepackage{color}
\usepackage{tikz}
\usepackage{multirow}
\usepackage{mathrsfs}
\usepackage{graphicx}
\usepackage{subfigure}
\newtheorem{thm}{Theorem}[section]

\newtheorem{prop}[thm]{Proposition}

\newtheorem{lem}[thm]{Lemma}
\newtheorem{example}[thm]{Example}
\newtheorem{false statement}{False statement}

\newtheorem{fact}[thm]{Fact}

\theoremstyle{definition}

\newtheorem{claim}[thm]{Claim}

\makeatletter \@addtoreset{equation}{section}

\def\hh{\mathcal{H}}

\def\hl{\mathcal{L}}
\def\hht{\mathcal{T}}

\def\hf{\mathcal{F}}
\def\hg{\mathcal{G}}

\def\hb{\mathcal{B}}
\def\hs{\mathcal{S}}

\begin{document}
\title{\bf\Large Intersecting families with large shadow degree}

\date{}
\author{Peter Frankl$^1$, Jian Wang$^2$\\[10pt]
$^{1}$R\'{e}nyi Institute, Budapest, Hungary\\[6pt]
$^{2}$Department of Mathematics\\
Taiyuan University of Technology\\
Taiyuan 030024, P. R. China\\[6pt]
E-mail:  $^1$frankl.peter@renyi.hu, $^2$wangjian01@tyut.edu.cn
}
\maketitle
\begin{abstract}
A $k$-uniform family $\mathcal{F}$ is called {\it intersecting} if $F\cap F'\neq \emptyset$ for all $F,F'\in \mathcal{F}$. The shadow family $\partial \mathcal{F}$ is the family of $(k-1)$-element sets that are contained in some members of $\mathcal{F}$.  The {\it shadow degree} (or {\it minimum positive co-degree}) of $\mathcal{F}$ is defined as the maximum integer $r$ such that every $E\in \partial \mathcal{F}$ is contained in at least $r$ members of $\mathcal{F}$. In 2021, Balogh, Lemons  and Palmer determined the maximum size of an intersecting $k$-uniform family with shadow degree at least $r$ for $n\geq n_0(k,r)$, where $n_0(k,r)$  is doubly exponential in $k$ for $4\leq r\leq k$. In the present paper,  we present a short proof of this result  for $n\geq 2(r+1)^rk \frac{\binom{2k-1}{k}}{\binom{2r-1}{r}}$ and $4\leq r\leq k$.
\end{abstract}

\section{Introduction}

Let $[n]$ be the standard $n$-element set $\{1,2,\ldots,n\}$. Let $2^{[n]}$  denote the power set of $[n]$ and let $\binom{[n]}{k}$ denote the collection of all $k$-element subsets of $[n]$. A subset $\hf$ of  $\binom{[n]}{k}$ is called a $k$-uniform family.  We call $\hf$ an {\it intersecting family} if $F\cap F'\neq \emptyset$ for all $F,F'\in \hf$.  The {\it matching number} $\nu(\hf)$ is the maximum number of disjoint sets in $\hf$. The {\it transversal number} $\tau(\hf)$ is the minimum size of $T\subset [n]$ such that $T\cap F\neq \emptyset$ for all $F\in \hf$.

One of the most important results in extremal set theory is the Erd\H{o}s-Ko-Rado theorem.

\begin{thm}[Erd\H{o}s-Ko-Rado \cite{ekr}]\label{thm-ekr}
Suppose that $n\geq 2k>0$, $\hf\subset\binom{[n]}{k}$ is intersecting, then
\begin{align}\label{ineq-ekr}
|\hf| \leq \binom{n-1}{k-1}.
\end{align}
\end{thm}

For $n>2k$, the equality holds in \eqref{ineq-ekr} if and only if $\hf=\{F\in \binom{[n]}{k}\colon x\in F\}=:\hs_x$ for some $x\in [n]$.

A subfamily of $\hs_x$ is called a {\it star}. Hilton and Milner proved   a strong stability result of the Erd\H{o}s-Ko-Rado theorem.

\begin{thm}[Hilton-Milner Theorem \cite{HM67}]
Suppose that $n> 2k\geq 4$, $\hf\subset\binom{[n]}{k}$ is an intersecting family that is not a star, then
\begin{align}\label{ineq-hm}
|\hf| \leq \binom{n-1}{k-1}-\binom{n-k-1}{k-1}+1.
\end{align}
\end{thm}

Let us define the Hilton-Milner family
\[
\hh(n,k) = \left\{H\in \binom{[n]}{k}\colon 1\in H,\ H\cap [2,k+1]\neq \emptyset\right\}\cup \{[2,k+1]\},
\]
showing that \eqref{ineq-hm} is best possible.

For $\hf\subset \binom{[n]}{k}$ and $E\subset[n]$, let
\[
\hf(E) =\{F\setminus E\colon E\subset F\in \hf\}.
\]
If $E=\{x\}$ then we simply write $\hf(x)$.
Define the {\it immediate shadow} $\partial \hf$ and the {\it shadow degree} (or {\it minimum positive co-degree}) $\delta_{k-1}^+(\hf)$ as
\[
\partial \hf =\left\{E\in \binom{[n]}{k-1}\colon E\subset F\in \hf\right\}\mbox{ and } \delta_{k-1}^+(\hf) =\min\{|\hf(E)|\colon E\in \partial \hf\}.
\]
It is easy to check that $\delta_{k-1}^+(\hs_x)=1$ and $\delta_{k-1}^+(\hh(n,k))=1$.

It is natural to ask for the maximum size of an intersecting family with shadow degree at least $r$.  Let us define the function:
\begin{align*}
&f(n,k,r)=\max\left\{|\hf|\colon  \hf\subset\binom{[n]}{k} \mbox{ with } \delta_{k-1}^+(\hf)\geq r\right\}.
\end{align*}

Note that Theorem \ref{thm-ekr} implies $f(n,k,1)=\binom{n-1}{k-1}$ for $n\geq 2k$.

\begin{example}
For $n\geq 2k$, $k\geq r\geq 1$ define
\[
\hl(n,k,r) =\left\{F\in \binom{[n]}{k}\colon |F\cap[2r-1]|\geq r\right\}.
\]
Clearly $\hl(n,k,r)$ is intersecting, $\delta_{k-1}^+(\hf)= r$ and
\[
|\hl(n,k,r)|=\sum_{r\leq i\leq 2r-1}\binom{2r-1}{i}\binom{n-2r+1}{k-i}.
\]
\end{example}

In 2021, Balogh, Lemons  and Palmer proved the following results.

\begin{prop}[\cite{BLP}]
If $\hf\subset \binom{[n]}{k}$ is a non-empty intersecting with $ \delta_{k-1}^+(\hf)\geq r$ then $r\leq k$.
\end{prop}

\begin{thm}[\cite{BLP}]\label{BLP}
Let $2\leq r\leq k\leq n$. If  $n\geq n_0(k,r)$ with $n_0(k,2)=\frac{k^4}{3}$, $n_0(k,3)=2k^5$ and $n_0(k,r)=((k+1)k^r2^k)^{2^k} \frac{(2k-2r)^{k-r}}{\binom{2r-1}{r}}$ for $r\geq 4$, then
\begin{align}\label{eq-fnkr}
f(n,k,r) =|\hl(n,k,r)|.
\end{align}
\end{thm}

In \cite{FW2024}, we proved Theorem \ref{BLP}  for $n_0(k,2)=28k$, $n_0(k,r)=Ck^{2(r+1)k}$ for $3\leq r< k$ and $n_0(k,k)=2k-1$.

In the present paper, by applying a result of Tuza \cite{tuza}, we give a short proof of Theorem \ref{BLP}, which provides a considerably improvement of $n_0(k,r)$ for $4\leq r<k$. Precisely, we prove the following result.

\begin{thm}\label{thm-main1}
Let $4\leq r\leq  k-1$. If  $n\geq 2(r+1)^rk \frac{\binom{2k-1}{k}}{\binom{2r-1}{r}}$, then
\begin{align}\label{eqnew-fnkr}
f(n,k,r) =  |\hl(n,k,r)|.
\end{align}
\end{thm}

Our proof works for the case $r=3$ as well. However in that case the bounds given by  Balogh, Lemons and Palmer \cite{BLP} are better.

We say that $\hh\subset 2^{[n]}$ is a {\it critical} intersecting family if $\hh$ is intersecting and  $(\hh\setminus \{H\})\cup (H\setminus \{x\})$ is no longer intersecting for any $H\in \hh$ and $x\in H$. The {\it rank} of $\hh$ is defined as $\max\{|H|\colon H\in \hh\}$.

We need the following result of Tuza.

\begin{thm}[Corollary 12 in \cite{tuza}]\label{thm-tuza}
Let $\hh$ be a critical intersecting family with rank $k$. Then
\[
|\cup_{H\in \hh} H| \leq \binom{2k-1}{k-1}+\binom{2k-4}{k-2}.
\]
\end{thm}

For $k=3$ and $r=2$, we determine $f(n,3,2)$ for the full range.

\begin{thm}\label{thm-main2}
For $n\geq 6$,
$f(n,3,2) =  |\hl(n,3,2)|$.
\end{thm}

\section{A slightly weaker bound for Tuza's result}

Even though Tuza's proof is not too long it is quite involved. Let us include here a short proof inspired by Katona \cite{Katona74} for a slightly weaker result.

\begin{prop}
Suppose that $\hh\subset 2^{[n]}$ is a critical intersecting hypergraph of rank $k$, $\cup \hh =[n]$. Then
\[
n\leq \binom{2k-3}{k-1}(2k-1).
\]
\end{prop}

\begin{proof}
For $1\leq i\leq n$ let us fix a pair $(G_i,H_i)$, $G_i,H_i\in \hh$  with $G_i\cap H_i=\{i\}$. Let $\pi=(x_1,\ldots,x_n)$ be an arbitrary permutation of $[n]$. We say that $\pi$ separates $(G_i,H_i)$ if letting $j$ be defined by $x_j=i$, either $G_i\subset \{x_1,\ldots,x_j\}$ and $H_i\subset \{x_j,\ldots,x_n\}$ or $H_i\subset \{x_1,\ldots,x_j\}$ and $G_i\subset \{x_j,\ldots,x_n\}$. Set $g=|G_i|$, $h=|H_i|$, $|G_i\cup H_i|=g+h-1$.

Whether $\pi$ separates $(G_i,H_i)$ depends only on the relative position of the elements of $G_i\cup H_i$. Easy computation shows that the number of $\pi$ separating $(G_i,H_i)$ is
\[
2\frac{n!(g-1)!(h-1)!}{(g+h-1)!} =\frac{2n!}{\binom{g+h-2}{g-1}(g+h-1)}\geq \frac{2n!}{\binom{2k-2}{k-1}(2k-1)}.
\]
Should the same $\pi$ separate $(G_i,H_i)$ and $(G_j,H_j)$ with $x_i$ preceding $x_j$ in $\pi$, either $G_i$ or $H_i$ completely precedes either $G_j$ or $H_j$, contradicting $\nu(\hh)=1$. Consequently,
\[
n\times \frac{2n!}{\binom{2k-2}{k-1}(2k-1)} \leq n!.
\]
It follows that
\[
|\cup \hh| =n \leq \binom{2k-3}{k-1} (2k-1).
\]
\end{proof}

\section{Proof of Theorem \ref{thm-main1}}

\begin{prop}\label{prop-fnkk2}
Let $\hh\subset \binom{[n]}{k}$ be a non-empty intersecting family with $\delta_{k-1}^+(\hh)\geq k$. Then $\hh=\binom{Y}{k}$ for some $Y\subset[n]$ with  $|Y|=2k-1$.
\end{prop}

\begin{proof}
Let us first prove the following claim.
\begin{claim}\label{claim-5}
To every $H\in \hh$ there exists $H'\in \hh$ with $|H\cap H'|=1$.
\end{claim}
\begin{proof}
To fixed $H\in \hh$ choose $G\in \hh$ with $|G\cap H|$ minimal. By $\nu(\hh)=1$, $G\cap H\neq \emptyset$. Say $t=|G\cap H|\geq 2$. Fix $G_0\in \binom{G}{k-1}$, $|G_0\cap H|=|G\cap H|-1=t-1\geq 1$. Then $|H\setminus G_0|=k-(t-1)<k$. Hence there exists $x\in [n]$ such that $G_0\cup \{x\}\in \hh$ and $x\notin H$. Thus
\[
|(G_0\cup \{x\})\cap H| =|G_0\cap H| <|G\cap H|,
\]
a contradiction.
\end{proof}

For notational convenience let $H_1=[k]$, $H_2=[k,2k-1]\in \hh$.

\begin{claim}\label{claim-6}
Suppose that $|G\cap H_i|=1$ ($i=1$ or 2) then $G\subset [2k-1]$.
\end{claim}
\begin{proof}
 By symmetry let $i=1$. Define $j$ by $G\cap H_1=\{j\}$. Using $|\hh([k+1,2k-1])|\geq k$ and $\nu(\hh)=1$, $\tilde{H} :=[k+1,2k-1]\cup \{j\}\in \hh$. If $\tilde{H}=G$ we are done. Otherwise $|\tilde{H}\cap G|<k$. Hence there exists $x\notin \tilde{H}\cap G$, $\tilde{H}_1:=([k]\setminus \{j\})\cup \{x\}\in \hh$. Then either $\tilde{H}_1\cap \tilde{H}$ or  $\tilde{H}_1\cap G$ is empty, a contradiction.
\end{proof}

To prove Proposition \ref{prop-fnkk2} assume indirectly that there exists $G\in \hh$ such that $G\not\subset [2k-1]$ and  choose $|G\cap H_1|$ to be minimal. By Claim \ref{claim-6} $|G\cap H_1|=:t\geq 2$. Fix $G_0\in \binom{G}{k-1}$ with $|G_0\cap H_1|=t-1$. Since $|H_1\setminus G_0|<k$ there is $x\notin H_1$ with $\tilde{G}:=G_0\cup \{x\}$ in $\hh$. Then $\tilde{G}\not\subset [2k-1]$, $|\tilde{G}\cap H_1|=t-1$, the final contradiction.
\end{proof}

Let $X$ be a finite set. For $\hh\subset 2^X$, let  $p(\hh)=\min\{|H|\colon H\in \hh\}$. For $p(\hh)\leq i\leq k$, define
\[
\hh^{(i)} =\left\{H\in \hh\colon |H|=i\right\}.
\]

\begin{lem}\label{lem-1}
Let $\hh\subset 2^X$ be  an intersecting family with $p(\hh)=p$. If $\tau(\hh^{(p)})\geq r$ then for $p\leq i\leq k$,
\[
|\hh^{(i)}| \leq p^r \binom{|X|-r}{i-r}.
\]
\end{lem}

\begin{proof}
We prove the lemma by a branching process. During the proof {\it a sequence} $S=(x_1,x_2,\ldots,x_\ell)$ is an ordered sequence of distinct elements of $X$ and we use $\widehat{S}$ to denote the underlying unordered set $\{x_1,x_2,\ldots,x_\ell\}$.
 At the first stage, we choose $H_1\in \hh^{(p)}$ and define $|H_1|$ sequences $(x_1)$, $x_1\in H_1$.

In each subsequent stage, we pick a sequence $S=(x_1,\ldots,x_\ell)$ with $\ell<r$. By $\tau(\hh^{(p)})\geq r$ one can choose $H_S\in \hh^{(p)}$ with $\widehat{S}\cap H_S=\emptyset$.  Then replace $S$ by the $p$ sequences $(x_1,\ldots,x_\ell,y)$ with $y\in H_S$. Eventually we shall obtain $p^r$ sequences of length $r$.

\begin{claim}
For each $H\in \hh$, there is a sequence $S$  with $\widehat{S}\subset H$.
\end{claim}

\begin{proof}
Let $S=(x_1,\ldots,x_\ell)$  be a sequence of maximal length that occurred
at some stage of the branching process satisfying $\widehat{S}\subset H$. Let us suppose indirectly that $\ell<r$.   Since  $H_1\cap H\neq \emptyset$ at the first stage, there is a sequence $(x_1)$ with $x_1\in H_1$ such that $\{x_1\}\subset H$. Thus $\ell\geq 1$. Since $\ell<r$, at some stage $S$ was picked and there is some  $H_S\in \hh^{(p)}$ with $\widehat{S}\cap H_S=\emptyset$ being chosen. Since $H\cap H_S\neq \emptyset$, there is some $y\in H\cap H_S$. Then $(x_1,\ldots,x_\ell,y)$ is a longer sequence satisfying $\{x_1,\ldots,x_\ell,y\}\subset H$, contradicting the maximality of $\ell$.
\end{proof}
Since there are only $\binom{|X|-r}{i-r}$ choices for the remaining elements of $H\in\hh^{(i)}$, the lemma is proven.
\end{proof}

\begin{lem}\label{lem-2}
Let $\hf\subset \binom{[n]}{k}$ be an intersecting family. Then $\tau(\hf)\geq \delta_{k-1}^+(\hf)$.
\end{lem}

\begin{proof}
Set $\delta=\delta_{k-1}^+(\hf)$ and let $D$ be an arbitrary set with $|D|=\delta-1$. We want to prove that $D\cap F=\emptyset$ for some $F\in \hf$. Argue indirectly pick an $F\in \hf$ with $0\neq |D\cap F|$ minimal. Choose $x\in D\cap F$ and consider $F\setminus \{x\}$. Note that $|D\setminus \{x\}|=\delta-2$ and there are at least $\delta-1$ choices for $y\neq x$ with $(F\setminus \{x\})\cup \{y\}=:F'\in \hf$. By choosing $y\notin D$, $|D\cap F'|<|D\cap F|$, a contradiction.
\end{proof}

Let $\hf\subset \binom{[n]}{k}$ be an intersecting family. Define the {\it trace} $\hf_{\mid X}$ of $\hf$ as the set $\{F\cap X\colon F\in \hf\}$.  We say that $X$ is a  {\it support}  of $\hf$ if  $\hf_{\mid X}$ is intersecting.

\begin{proof}[Proof of Theorem \ref{thm-main1}]
Let $\hf\subset \binom{[n]}{k}$ be an intersecting family with $\delta_{k-1}^+(\hf)\geq r$. Let $X$ be a support of $\hf$ of minimum size. Let $\hh=\hf_{\mid X}$ and let $\hb\subset \hh$ be the family of minimal (for containment) members of $\hh$. It is easy to check that $\hb$ is a critical intersecting family and $\cup_{B\in \hb} B=X$. Then by Theorem \ref{thm-tuza}
\begin{align}\label{ineq-key1}
|X| =|\cup_{B\in \hb} B| \leq \binom{2k-1}{k-1}+\binom{2k-4}{k-2}\leq \binom{2k}{k-1}.
\end{align}

Let $p=p(\hh)$.

\begin{claim}\label{claim-hph}
$\hh^{(p)}$ is intersecting with $\delta_{p-1}^+(\hh^{(p)})\geq r$.
\end{claim}
\begin{proof}
Since $X$ is a support, $H\cap H'\neq \emptyset$ for $H,H'\in \hh^{(p)}$. Thus $\hh^{(p)}$ is intersecting.

Fix $F\in \hf$ with $H=F\cap X\in \hh^{(p)}$ and let $R\in \binom{H}{p-1}$. Since $\delta_{k-1}^+(\hf)\geq r$, there are $r$ distinct elements $x_1,x_2,\ldots,x_r$ such that $F_j':=((F\setminus X)\cup R\cup \{x_j\})\in \hf$.  If $x_j\notin X$ then $F_j'\cap X=R\in \hh$ and $|R|=p-1$, contradicting the minimality of $p$. Thus $x_j\in X$ for all $j=1,2,\ldots,r$, proving $\delta_{p-1}^+(\hh^{(p)})\geq r$.
\end{proof}

%

Note that $F\cap X$ is a transversal of $\hf$ for any $F\in \hf$. By Lemma \ref{lem-2} we infer that $p\geq \tau(\hf)\geq \delta^+_{k-1}(\hf)\geq r$.
If $p=r$, then by Claim \ref{claim-hph} and Proposition \ref{prop-fnkk2} we infer that $\hh^{(p)}=\binom{Y}{r}$ for some $Y\in \binom{X}{2r-1}$. We claim that $|F\cap Y|\geq r$ for all $F\in \hf$. Indeed the opposite would mean that $F\cap H=\emptyset$ for some $H\in \binom{Y}{r}$. Choose  $E\in \binom{[n]\setminus X}{k-r}$ such that $H\cup E\in \hf$.  Then $F\cap (H\cup E)\cap X=F\cap H=\emptyset$, contradicting the fact that $X$ is a support. Consequently $\hf\subset\{F\in \binom{[n]}{k}\colon |F\cap Y|\geq r\}$, i.e., $\hf$ is contained in an isomorphic copy of $\hl(n,k,r)$. Thus we may assume that $p\geq r+1$.

By Claim \ref{claim-hph} and Lemma \ref{lem-2}, we have $\tau(\hh^{(p)})\geq r$. Then by Lemma \ref{lem-1},
\[
|\hf|\leq \sum_{p\leq i\leq k} |\hh^{(i)}| \binom{n-|X|}{k-i} \leq \sum_{p\leq i\leq k} p^r \binom{|X|-r}{i-r} \binom{n-|X|}{k-i}.
\]
Note that
\[
\frac{\binom{|X|-r}{i-r} \binom{n-|X|}{k-i}}{\binom{|X|-r}{i+1-r} \binom{n-|X|}{k-i-1}} = \frac{i+1-r}{|X|-i} \cdot \frac{n-|X|-k+i+1}{k-i}.
\]
Since both $\frac{i+1-r}{|X|-i}$ and $\frac{n-|X|-k+i+1}{k-i}$ are monotone increasing functions of $i$, plugging in $i=r+1$ we obtain
\[
\frac{2}{|X|-r-1}\cdot \frac{n-|X|-k+r+2}{k-r-1} > \frac{2(k-1)(|X|-1)}{(k-3)(|X|-1)}>2
\]
for $n\geq k\binom{2k}{k-1}>k|X|$. Thus,
\[
|\hf|\leq \sum_{p\leq i\leq k} |\hh^{(i)}| \binom{n-|X|}{k-i} < 2 p^r \binom{|X|-r}{p-r} \binom{n-|X|}{k-p}.
\]
Since $p\geq r+1>r$ implies $\left(\frac{p}{p+1}\right)^r\geq \left(\frac{r}{r+1}\right)^r\geq \frac{1}{e}$, for $n\geq (2k+1)\binom{2k}{k-1}>(2k+1)|X|$ we have
\begin{align*}
\frac{p^r\binom{|X|-r}{p-r} \binom{n-|X|}{k-p}}{(p+1)^r\binom{|X|-r}{p+1-r} \binom{n-|X|}{k-p-1}} \geq \frac{1}{e}\cdot\frac{2}{|X|-r-1}\cdot \frac{n-|X|-k+r+2}{k-r-1}\geq \frac{ 4k(|X|-1)}{ek(|X|-1)}> 1.
\end{align*}
By $p\geq r+1$, it follows that
\[
|\hf| \leq  2p^r \binom{|X|-r}{p-r} \binom{n-|X|}{k-p}<2(r+1)^r (|X|-r) \binom{n-|X|}{k-r-1}.
\]
By the minimality of $X$, for any $x\in X$ there exists $H_1,H_2\in \hh$ such that $H_1\cap H_2=\{x\}$. It follows that $|X|\geq |H_1\cup H_2|\geq 2r+1$. Thus, for $n\geq 2(r+1)^rk \frac{\binom{2k-1}{k}}{\binom{2r-1}{r}}$ we have
\begin{align*}
|\hf| <2(r+1)^r (|X|-r) \binom{n-2r}{k-r-1}\leq \binom{2r-1}{r}\binom{n-2r+1}{k-r}< |\hl(n,k,r)|.
\end{align*}
\end{proof}

\section{Proof of Theorem \ref{thm-main2}}

A collection of sets $F_1,\ldots,F_\ell$ is called a {\it sunflower of size $\ell$ with  kernel $C$} if $F_i\cap F_j=C$ for all distinct $i,j\in \{1,2,\ldots,\ell\}$.

\begin{prop}\label{prop-1}
Suppose that $\hf\subset \binom{[n]}{k}$ is intersecting and $\delta_{k-1}^+(\hf)\geq 2$. Then $\hf$ contains no sunflower of size 3 and kernel of size 1.
\end{prop}
\begin{proof}
Let $F_1,F_2,F_3\in \hf$ with $F_i\cap F_j=\{x\}$ for $1\leq i<j\leq 3$. Consider $G\in \hf$, $G\neq F_3$ with $F_3\setminus \{x\} \subset G$. Then $G\cap F_1\neq \emptyset$ and $G\cap F_2\neq \emptyset$ imply $|G|\geq k+1$, a contradiction.
\end{proof}


\begin{fact}\label{fact-3.3}
If $\delta_{k-1}^+(\hf)\geq r$ then $\delta_{k-2}^+(\hf(x))\geq r$ for any $x\in [n]$.
\end{fact}

\begin{proof}
For any $E\in \hf(x)$ and $R\in \binom{E}{k-2}$, $R\cup \{x\}$ is covered by at least $r$ sets in $\hf$. Since all these sets contain $x$, it follows that $\delta_{k-2}^+(\hf(x))\geq r$.
\end{proof}

For $Y\in \binom{[n]}{2r-1}$, define
\[
\hl(n,k,Y) =\left\{F\in \binom{[n]}{k}\colon |F\cap Y|\geq r\right\}.
\]
For $\hg\subset \binom{[n]}{2}$, we use $\delta(\hg)$ to denote the minimum degree of $\hg$.

\begin{proof}[Proof of Theorem \ref{thm-main2}]
Let $\hht\subset \binom{[n]}{3}$ be an intersecting family with $\delta_2^{+}(\hht)\geq 2$. Let $X\subset [n]$ be a minimal support and let $y\in [n]\setminus X$. By Fact \ref{fact-3.3} we infer that $\hht(y)$ is a graph with minimum degree at least 2.

Since $X$ is a minimal support, for any $T,T'\in \hht$,
\begin{align}\label{ineq-0}
T\cap T'=\{z\} \mbox{ implies } z\in X.
\end{align}
Since $y\notin X$, by \eqref{ineq-0} we infer that $\hht(y)$ is intersecting. By Proposition \ref{prop-fnkk2} we have either $\hht(y)=\emptyset$ or $\hht(y)=\binom{Y}{2}$ for some $Y$ with $|Y|=3$.

\begin{claim}\label{claim-1}
Let $z\in [n]$.
If $\hht(z)=\binom{Y}{2}$ for some $Y$ with $|Y|=3$, then $\hht\subset \hl(n,3,Y)$.
\end{claim}
\begin{proof}
Choose $T\in \hht$. If $z\in T$ then $|T\cap Y|=2$ by definition. If $z\notin T$ then by $\nu(\hht)=1$, $T\cap P\neq \emptyset$ for each pair $P\in \binom{Y}{2}$. Hence $|T\cap Y|\geq 2$.
\end{proof}

Claim \ref{claim-1} shows that $|\hht| \leq |\hl(n,3,2)|$ unless $\hht(y)=\emptyset$ for all $y\in [n]\setminus X$.  Thus we may assume that $X=[n]$. Since $X$ is a support of minimal size, for each $x\in [n]$ there exist $T_x,T_x'\in \hht$ with $T_x\cap T_x'=\{x\}$. In particular $\nu(\hht(x))\geq 2$. By proposition \ref{prop-1} equality holds.

%
%


\begin{claim}\label{claim-2}
$\partial \hht=\binom{[n]}{2}$ and  the vertex set of $\hht(x)$ is $[n]\setminus \{x\}$ for any $x\in [n]$.
\end{claim}
\begin{proof}
 Suppose that there exist $x,y\in [n]$ with $x\neq y$ such that $\{x,y\}\notin \partial \hht$. Let $\{u_1,u_2\},\{v_1,v_2\}\in \hht(x)$ be a matching. By $\nu(\hht)=1$, $\hht(x),\hht(y)$ are cross-intersecting.  It follows that $\hht(y)\subset \{\{u_i,v_j\}\colon i=1,2, j=1,2\}$. Since $\delta(\hht(x))\geq 2$, $\hht(y)= \{\{u_i,v_j\}\colon i=1,2, j=1,2\}$. Then by the cross-intersecting property we infer that $\hht(x)=\{\{u_1,u_2\},\{v_1,v_2\}\}$, contradicting $\delta(\hht(y))\geq 2$. Thus $\partial \hht=\binom{[n]}{2}$. It follows that the vertex set of $\hht(x)$ is $[n]\setminus \{x\}$.
\end{proof}

For $n=6$ the Erd\H{o}s-Ko-Rado theorem implies $|\hht|\leq \binom{6-1}{3-1}=10= 3n-8$. If $|\hht|=10$ then $|\hht|\binom{3}{2}=2\binom{6}{2}$. Hence every 2-subset of $[6]$ is contained in exactly two edges of $\hht$. There is a unique 2-design with these properties (cf. e.g. \cite{Furedi83}). Summarizing, for $n=6$ there are three distinct triple-systems achieving equality. Namely, $\binom{[5]}{3}$, $\hl(6,3,2)$ and the above 2-design.

From now on assume $n\geq 7$.

 \begin{claim}\label{claim-3}
For any $x\in [n]$, $\hht(x)$ contains no $C_{2\ell+1}$ for all $\ell\geq 2$.
\end{claim}

\begin{proof}
Since $\nu(\hht(x))=2$, $\hht(x)$ contains no $C_\ell$, $\ell\geq 6$.   Also if it contains $C_5$ then $\hht(x)$ has exactly 5 vertices. Say $x=1$, $C_5$ is on $[2,6]$. By $n\geq 7$  we consider $\hht(7)$. Now $\tau(C_5)=3$ implies the existence of $(w,z)\in C_5$ with $\{u,v,7\}\cap \{1,w,z\}=\emptyset$, contradicting $\nu(\hht)=1$.
\end{proof}

\begin{claim}\label{claim-4}
For any $x\in [n]$,  $\hht(x)$ is a $K_{2,n-3}$.
\end{claim}

\begin{proof}
First let us show that $\hht(x)$ is bipartite.  By Claim \ref{claim-3} we are left to show that $\hht(x)$ is triangle-free. Indeed, otherwise let $y_1y_2y_3$ be a triangle in $\hht(x)$. By Claim \ref{claim-2}, $\{u,v\}\in \partial \hht$ for any $\{u,v\}\in [n]\setminus \{y_1,y_2,y_3,x\}$. By $\delta_2^+(\hht)\geq 2$, we can choose $w\in [n]\setminus \{x\}$ such that $\{u,v,w\}\in \hht$. Then one of $\{x,y_1,y_2\}$, $\{x,y_2,y_3\}$, $\{x,y_1,y_3\}$ is disjoint to $\{u,v,w\}$, contradicting $\nu(\hht)=1$.
  Thus $\hht(x)$ is bipartite.


By K\"{o}nig-Hall Theorem all the edges of $\hht(x)$ are covered by two vertices, say $\{y,z\}$. Should they be in different partite sets, we get a contradiction with $\delta(\hht(x))\geq 2$. Thus $\{y,z\}$ is one of the partite classes. Using $\delta(\hht(x))\geq 2$ again, the bipartite graph $\hht(x)$ must be $K_{2,n-3}$.
\end{proof}

Let $\hht(x)=\{\{u_i,v_j\}\colon i=1,2, j=1,2,\ldots,n-3\}$. Then $\{x,u_1\},\{x,u_2\}\in \hht(v_1)$. Since $\delta_2^+(\hht)\geq 2$, there is some $\{v_1,u_1,z\}\in \hht$ such that $z\neq x$. Since $\{v_1,u_1,z\} \cap\{x,u_2,v_j\}\neq \emptyset$ for all $j=2,3,\ldots,n-3$, we must have $z=u_2$. Thus $\{\{x,u_1\},\{x,u_2\},\{u_1,u_2\}\}\subset \hht(v_1)$, contradicting Claim \ref{claim-4}. Thus the theorem is proven.
%
\end{proof}

\vspace{6pt}\noindent
{\bf Acknowledgement:}  The first author's research was partially supported by the National Research, Development and Innovation Office NKFIH, grant K132696.

\end{document}